\documentclass[11pt]{article}
\usepackage[margin=1.1in]{geometry}
\usepackage{amsmath,amssymb,amsthm}

\usepackage{tikz}
\usetikzlibrary{calc, arrows.meta}

\usepackage[colorlinks=true,linkcolor=blue,citecolor=blue]{hyperref}

\newtheorem{theorem}{Theorem}[section]

\newtheorem{observation}[theorem]{Observation}
\newtheorem{definition}[theorem]{Definition}
\newtheorem{corollary}[theorem]{Corollary}
\newtheorem{conjecture}[theorem]{Conjecture}
\theoremstyle{remark}

\theoremstyle{plain}

\def\Z{\mathbb{Z}}
\newcommand{\eps}{\varepsilon}

\newcommand{\calC}{\mathcal C}
\newcommand{\calI}{\mathcal I}
\newcommand{\calN}{\mathcal N}
\newcommand{\CC}{\mathbb C}
\newcommand{\ip}[2]{\langle #1,\, #2\rangle}

\newcommand{\symdiff}{\mathbin{\triangle}}

\newcommand{\doi}[1]{\href{https://doi.org/#1}{doi:\,#1}}
\newcommand{\pdflink}[1]{\href{#1}{#1}}

\hypersetup{pdfauthor={Radek Hušek, Robert Šámal}}

\title{Exponentially Many Circuit Double Covers\footnote{Supported by grant 25-16627S of the Czech Science Foundation.}}
\author{%
Radek Hušek\thanks{%
\texttt{radek.husek@fit.cvut.cz}, Faculty of Information Technology, Czech Technical University in Prague,
Czech Republic
}%
\and Robert Šámal\thanks{
\texttt{samal@iuuk.mff.cuni.cz},
Computer Science Institute of Charles University,
Faculty of Mathematics and Physics, Charles University,
Prague, Czech Republic
}%
}
\date{} 

\begin{document}
\maketitle

\begin{abstract}
  The cycle double cover conjecture of Szekeres and Seymour, the proof of which was recently announced
  by OpenAI, states that every bridgeless graph has a collection of cycles
  covering every edge exactly twice. 
  We study the counting version of this statement for cubic graphs, where we count circuit double covers ---
  collections of circuits (connected 2-regular subgraphs) covering every edge
  twice. We show that every 2-edge-connected 3-edge-colorable cubic graph on $n$
  vertices has at least $2^{n/2-1}$ circuit double covers, matching our
  previously conjectured general lower bound. 
  For every 3-edge-connected cubic graph with girth at least 16 we show a weaker
  exponential lower bound on circuit double covers.
  For both of these results we use the same system of linear equations used by OpenAI in their proof, 
  however, we provide additional combinatorial interpretation. 
  We characterize planarity of a cubic graph by solvability of this system of equations for arbitrary nowhere-zero $\Z_2^k$-flow. 
  We give a condition on the flow that is equivalent to existence of a 5-cycle double cover.
\end{abstract}

\section{Introduction}\label{sec:intro}

We consider undirected connected graphs. The graph is usually denoted $G$ with
$V$ being the set of its vertices and $E$ the set of its edges.
We consider simple cubic (i.e., 3-regular) graphs unless noted otherwise.
Loops and multiedges are not interesting as they
force existence of a 1-cut or 2-cut unless the connected component
is just three parallel edges.
Throughout, connectivity and cuts refer to edge-connectivity and edge-cuts unless explicitly stated otherwise.
However, for cubic graphs, the edge and vertex connectivity coincide. 

We denote by $E(v)$ the set of edges incident with vertex $v$.
We write $[A \to B]$ for the set of all functions from~$A$ to~$B$.
We will use flows throughout the paper but always only $\Z_2^k$-flows,
so we do not need directed edges. Note that we always treat $\Z_2^k$
as a $k$-dimensional space over $\Z_2$.

\begin{definition}[Flow]
  A mapping $f: E \to \Z_2^k$ is a {\em $\Z_2^k$-flow} if $\sum_{e \in E(v)} f(e) = 0$ for
  all $v \in V$. We say that a flow $f$ is {\em nowhere-zero} (or {\em NZ} for short)
  if $f(e) \neq 0$ for all $e \in E$.
\end{definition}

Two related notions of ``double cover'' appear in the literature under names that are
often used loosely.

\begin{definition}[Cycle double cover]\label{def:cydc}
A subgraph of a graph $G$ is \emph{even} if every vertex has even degree in it (the
empty subgraph is allowed). A \emph{(labeled) $k$-cycle double cover} ($k$-CyDC) of $G$
is a tuple\footnote{I.e., the order of the cycles matters.} $(C_1, \dots, C_k)$
of even subgraphs of $G$ such that every edge of $G$
belongs to exactly two members. We say just CyDC if we do not care about $k$.
\end{definition}

Edge sets of even subgraphs are exactly cycles of the matroid theory. 
For this reason (and following large part of relevant literature) we will 
call even subgraphs cycles in the rest of the paper. 
If the number of cycles in a CyDC is a power of two, we will identify their
labels with elements of $\Z_2^p$ for a suitable $p$.
We treat cycles (and also circuit defined below) as sets of edges.
For cycle $C$ we denote $V(C) = \{ v \in V : E(v) \cap C \neq \emptyset \}$
the set of its vertices and $\gamma(C): E \to \Z_2$ the flow defined
$\gamma(C)(e) = 1$ if $e \in C$ and 0 otherwise.
The \emph{cycle space} of a graph is the subspace of $\Z_2^E$
consisting of the characteristic vectors of all cycles.

Note that we fix the order of cycles in CyDC so two of them differing in the order
of cycles are considered to be different. This might be slightly non-standard but
it fits better the presented approach and the difference to unordered version
is at most $k!$ for counting $k$-CyDCs. Because we will consider at most 8-CyDCs,
this is a constant.

\begin{definition}[Circuit double cover]\label{def:cidc}
A \emph{circuit} of $G$ is a connected $2$-regular subgraph. A
\emph{circuit double cover} (CiDC) of $G$ is a finite family of
circuits of $G$ such that every edge of $G$ belongs to exactly two members of the
family.
\end{definition}

In a general graph, a CiDC may need to be a multiset, but in cubic graphs we 
will never need to use the same circuit twice. 
Since every circuit is a cycle, every CiDC is a CyDC.
Conversely, every cycle can be decomposed into a set of circuits.
Hence they are the same for existential questions but differ for counting, obviously.

\begin{conjecture}[Szekeres~\cite{szekeres}, Seymour~\cite{seymour}]\label{conj:CDC}
Every bridgeless cubic graph has a CyDC (equivalently, a CiDC). 
\end{conjecture}

In our earlier paper we conjectured (based on numerical experiments and proofs in special cases) that there is 
always an exponential number of circuit double covers. We ask for those for two reasons. 
First, their number has perhaps more natural interpretation, as they correspond to surface embeddings 
(every circuit is one face boundary). Second, as we argued in~\cite{hs}, lower bounds on 
CyDC are easier to obtain, and the following (stronger) conjecture appears to be true. 

\begin{conjecture}[Hušek, Šámal~\cite{hs}]\label{conj:manyCiDCs}
Every 2-connected cubic graph on $n$ vertices has at least $2^{n/2 - 1}$ CiDCs.
\end{conjecture}

Conjecture~\ref{conj:CDC} was recently proven by OpenAI~\cite{openai}.\footnote{As of
July 2026 the proof has not yet appeared in a refereed publication, but it is
accompanied by a Lean formalization~\cite{openai-lean}.} 
An exposition by Oum~\cite{Oum} adds more explanations to the proof. 
We present another point of view on the proof announced by OpenAI and extend it in several ways, 
mostly involving counting. 
First, we prove Corollary~\ref{cor:8CyDC}, that gives $\Omega(2^{n/2})$ 8-CyDCs by 
only noting that 8-CyDCs bijectively correspond to solutions of the system of linear equations~\eqref{eq:cond}. 
We also show the following linear algebraic planarity characterization.
\begin{theorem}\label{thm:planchar}
  Let $G$ be a 2-edge-connected cubic graph. Then $G$ is planar if and only if system~\eqref{eq:cond} 
  is solvable for every nowhere-zero $\Z_2^5$-flow on $G$. If $G$ is 3-edge-colorable, 
  $\Z_2^5$ may be replaced by $\Z_2^4$.
\end{theorem}

For 3-edge-colorable cubic graphs (where existence of a CiDC is well-known), we prove
a bound on their number matching Conjecture~\ref{conj:manyCiDCs}. 

\begin{theorem}\label{thm:3-col}
Let $G$ be a 2-connected 3-edge-colorable cubic graph on $n$ vertices.
Then $G$ has at least $2^{n/2 - 1}$ CiDCs. 
\end{theorem}

In \cite{hs} we observed the result is tight for so-called Klee graphs. 
In \cite{hs} we showed that the minimal counterexample to Conjecture~\ref{conj:manyCiDCs}
does not have nontrivial cut smaller than 4 and does not have 4-cycles.
Here, we prove that graphs with high enough girth have exponentially many CiDCs.

\begin{theorem}\label{thm:snarks}
Let $G$ be a 3-connected cubic graph on $n$ vertices with girth $g\ge 16$.
Then the number of CiDCs of $G$ is $2^{\Omega(n)}$.
\end{theorem}

While the proof of Conjecture~\ref{conj:CDC} is certainly big news, there are 
still many versions that remain open. In particular, 
it's been conjectured that five cycles suffice for each graph.

\begin{conjecture}[Celmins \cite{Celmin}, Preissmann \cite{Preismann}]\label{conj:5CDC}
  Every bridgeless cubic graph has a 5-CyDC.
\end{conjecture}

We don't know if the present methods can be extended to prove this, but in Section~\ref{sec:5CDC} 
we provide some ideas in this direction. 

\paragraph{A technical overview of the paper}
In Section~\ref{sec:constr} we understand a solution of~\eqref{eq:cond} as a properly face-labeled signed embedding; equivalently, it is a labeled CyDC.
For $k=2$, a 3-edge-coloring makes the system homogeneous and leaves a kernel of dimension at least $n/2+1$; this is done
in Section~\ref{sec:kfour}. 
For $k=3$, every NZ flow admits a solution. There are exponentially many flows, hence exponentially many labeled $8$-CyDCs.
Several labeled CyDCs can induce the same CiDC. A CiDC with $c$ circuits admits only $O(6^c)$ proper 8-colorings; girth more than $g$ gives $c\le 3n/g$.
A $5$-CyDC is obtained by restricting the allowable face labels to a five-element set, and this becomes a parity condition on a chosen $\Z_2^3$-flow.

\paragraph{Relation to~\cite{openai}}
This paper builds on the proof of CyDC conjecture by~\cite{openai}. The construction
itself, Observations~\ref{obs:sol-to-cydc}, \ref{obs:annih}, \ref{obs:annih-vert},
\ref{obs:annih-3} and~\ref{obs:align} are reformulations of their proof from a more
combinatorial point of view. The counting results, planarity characterization
and combinatorial insights are new.

\section{The construction}
\label{sec:constr}

The technique of modifying nowhere-zero $\Z_2^k$-flows to get a CyDC started with 
Tutte's~\cite{tutte49} correspondence between nowhere-zero $\Z_2^2$-flows and 4-colorable CyDCs, 
which are a special case of~\eqref{eq:cond}. 
The key novelty of OpenAI's approach~\cite{openai} is adding right-hand side to these equations. 
The following construction is a reformulation of the proof by OpenAI but we give many more details about it
and we derive the equations using the language of graph embeddings. 

We fix a nowhere-zero flow~$f : E \to \Z_2^k$ for some $k \geq 2$. 
We fix for every vertex~$v$ a 3-cycle $\rho_{v} : E(v) \to E(v)$, in an arbitrary way. (Usually, independently of~$f$ and of~$k$). 
For a vertex-edge incidence $(v,e)$, call the side opposite $\rho_v(e)$ the $0$-side of $e$ at $v$, and the side opposite $\rho_v^{-1}(e)$ its $1$-side.
We define their labels as in Figure~\ref{fig:rho}, using variables $t_v \in \Z_2^k$: 
\begin{equation}\label{eq:labels}
  a^0_{v,e}:=t_v+f(\rho_v(e)),\qquad a^1_{v,e}:=t_v+f(\rho_v^{-1}(e)) 
\end{equation}
For $e=uv$, we identify side $i$ of $e$ at $u$ with side $i+\eps_e$ at $v$, with a binary variable $\eps_e$. 
This defines an embedding scheme of~$G$, as defined, e.g., in Mohar and Thomassen~\cite{MTbook}, 
we let the two facial walks through edge~$e$ cross in the middle of~$e$, whenever $\eps_e = 0$
(see Figure~\ref{fig:eps}, the typical use of sign $\pm 1$ is less convenient here). 
The next observation shows that the labeling of edge-sides (also called darts, or half-edges in the literature) can be 
consistently extended to a labeling of face boundaries -- whenever a system of linear equations is satisfied. 

\begin{observation}\label{obs:keyemb}
  With the notation above and definition~\eqref{eq:labels} we have
\begin{enumerate}
  \item (Vertex consistency) The $1$-side of $e$ and the $0$-side of $\rho_v(e)$ have the same label. 
  \item (Edge consistency) $a^j_{u,e}=a^{j+\eps_e}_{v,e}$ is equivalent to 
    \begin{align}
      \left.
      \begin{aligned}
        t_u + t_v + \eps_{uv} f(uv) &= f(\rho_{u}(uv)) + f(\rho_{v}(uv))  &&uv \in E \\
        t_v &\in \Z_2^k                                                   &&v \in V \\
        \eps_e &\in \Z_2                                                  &&e \in E
      \end{aligned}
      \qquad \right\} \label{eq:cond}
    \end{align}
  \item (Edge difference) The two labels on~$e$ differ by~$f(e)$. 
\end{enumerate}
\end{observation}

\begin{proof}
  1. We need to show $a^1_{v,e} = a^0_{v,\rho_v(e)}$. By~\eqref{eq:labels}, both are equal to $t_v + f(\rho^2_v(e))$. 

  2. Flow condition at $v$ gives $f(e) + f(\rho_v(e)) + f(\rho^{-1}_v(e)) = 0$, together with~\eqref{eq:labels}, gives 
  us~\eqref{eq:cond}. 

  3. The flow condition implies $a^1_{v,e} =a^0_{v,e}+f(e)$. 
\end{proof}

Note, that replacing $\rho_v$ by $\rho_v^{-1}$ and simultaneously toggling $\eps_e$ on all $e\in E(v)$ is the usual switching operation on an
embedding scheme and does not change the resulting embedding.
In the rest of this section we will comment on the bijection between CyDCs and solutions to~\eqref{eq:cond}. 
The hard part though is to find whether the system is solvable at all and we will answer
it in Section~\ref{sec:appl} separately for $k=2$ and $k = 3$.

    \begin{figure}[h]
      \centering
      \includegraphics[page=2]{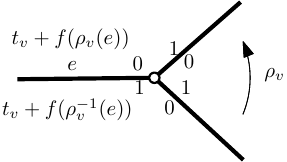}
      \caption{Labels $a^0_{v,e}$ and $a^1_{v,e}$.}
      \label{fig:rho}
    \end{figure}

    \begin{figure}[htb]
      \centering
      \begin{tikzpicture}[
          vertex/.style={circle, fill=black, inner sep=0pt, minimum size=5pt},
          edge/.style={very thick},
          path1/.style={blue, thick, rounded corners=4pt},
          path2/.style={red, thick, rounded corners=4pt},
          rotation/.style={semithick, -{Stealth[length=4pt, width=4pt]}}
        ]

        \begin{scope}[shift={(8,0)}]
          \node[vertex, label={left:$u$}] (u) at (0,0) {};
          \node[vertex, label={right:$v$}] (v) at (3,0) {};

          \draw[edge] (u) -- (v) node[midway, below=4pt] {$e$};
          \draw[edge] (u) -- ++(-1, 1) node[left] {$a$};
          \draw[edge] (u) -- ++(-1, -1) node[left] {$b$};
          \draw[edge] (v) -- ++(1, 1) node[right] {$y$};
          \draw[edge] (v) -- ++(1, -1) node[right] {$x$};

          \draw[path1] (-1, 1.2) -- (0, 0.2) -- (3, 0.2) -- (4, 1.2);
          \draw[path2] (-1, -1.2) -- (0, -0.2) -- (3, -0.2) -- (4, -1.2);

          \draw[rotation] ($(u)+(45:0.72)$)
            arc[start angle=45, end angle=90, radius=0.72]
            node[pos=0.52, above right=-1pt] {$\rho_u$};
          \draw[rotation] ($(v)+(225:0.72)$)
            arc[start angle=225, end angle=270, radius=0.72]
            node[pos=0.52, below left=-1pt] {$\rho_v$};

          \node[font=\large] at (1.5, -2) {$\eps_e = 1$};
        \end{scope}

        \begin{scope}[shift={(0,0)}]
          \node[vertex, label=left:$u$] (u) at (0,0) {};
          \node[vertex, label=right:$v$] (v) at (3,0) {};

          \draw[edge] (u) -- (v) node[midway, below=4pt] {$e$};
          \draw[edge] (u) -- ++(-1, 1) node[left] {$a$};
          \draw[edge] (u) -- ++(-1, -1) node[left] {$b$};
          \draw[edge] (v) -- ++(1, 1) node[right] {$y$};
          \draw[edge] (v) -- ++(1, -1) node[right] {$x$};

          \draw[path1] (-1, 1.2) -- (0, 0.2) -- (1.1, 0.2) -- (1.9, -0.2) -- (3, -0.2) -- (4, -1.2);
          \draw[path2] (-1, -1.2) -- (0, -0.2) -- (1.1, -0.2) -- (1.9, 0.2) -- (3, 0.2) -- (4, 1.2);

          \draw[rotation] ($(u)+(45:0.72)$)
            arc[start angle=45, end angle=90, radius=0.72]
            node[pos=0.52, above right=-1pt] {$\rho_u$};
          \draw[rotation] ($(v)+(225:0.72)$)
            arc[start angle=225, end angle=270, radius=0.72]
            node[pos=0.52, below left=-1pt] {$\rho_v$};

          \node[font=\large] at (1.5, -2) {$\eps_e = 0$};
        \end{scope}
      \end{tikzpicture}
      \caption{Interpretation of $\eps_e$.}\label{fig:eps}
    \end{figure}
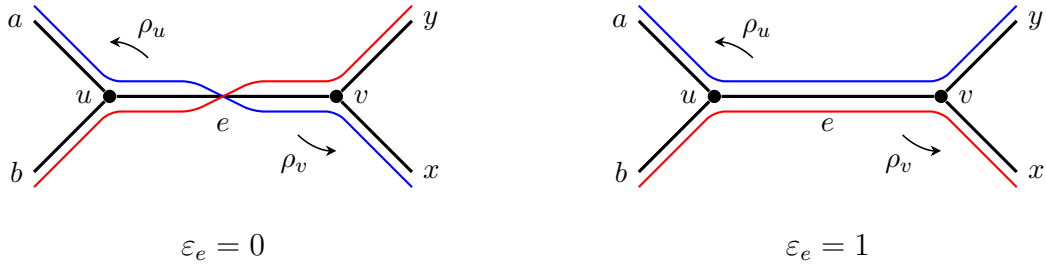

\begin{definition}
  Let $\mathcal F_k \subset [E \to \Z_2^k]$ be the set of all NZ $\Z_2^k$-flows on $G$.
  Let $\mathcal S_k = (\Z_2^k)^V \times \Z_2^E$ be the possible assignments of
  values to $t$ and $\eps$ in system (\ref{eq:cond}).
  Define $$\calN_k = \{ (f, s) \in \mathcal F_k \times \mathcal S_k : \textrm{$s$ is
  a valid solution of (\ref{eq:cond}) for flow $f$} \}.$$
  Let $\CC_{2^k} \subset [\Z_2^k \to \Z_2^E]$ be the set of all $2^k$-CyDCs of $G$.
  Define $\mu_k: \calN_k \to \CC_{2^k}$ as the mapping
  that assigns every flow and every valid solution of (\ref{eq:cond})
  a $2^k$-CyDC via formula 
  $$ 
    C_i = \{ uv \in E : a^0_{v,uv} = i \mbox{ or }  a^1_{v,uv} = i\}
  $$ 
  for $i \in \Z_2^k$.
\end{definition}


\begin{observation}\label{obs:sol-to-cydc}
  For every $(f, s) \in \calN_k$, $\mu_k(f, s)$ is a valid $2^k$-CyDC.
\end{observation}

\begin{proof}
  We will use Observation~\ref{obs:keyemb}. Part~2 shows that the definition of~$C_i$ is 
  independent of the choice of end vertex. Part~3 shows we get a double cover, and 
  part~1 ensures each~$C_i$ is a cycle. 
\end{proof}

\begin{observation}\label{obs:cydc-to-sol}
  The mapping $\mu_k: \calN_k \to \CC_{2^k}$ is a bijection.
\end{observation}

\begin{proof}
  We prove that it is injective and onto.
  Onto first, for every $\calC \in \CC_{2^k}$ we need to find $(f, s) \in \calN_k$
  so that $\mu_k(f, s) = \calC$.
  Let $(C_i)_{i \in \Z_2^k} = \calC$ be the cycles of $\calC$.

  Define $f = \sum_{i \in \Z_2^k} i \cdot \gamma(C_i)$, i.e., $f$~is the linear combination
  of (the flows defined by) the cycles with coefficients being their labels.
  Note that 
  \begin{equation}
    \text{$f(e) = i + j$ for the unique $i \neq j$ such $e \in C_i \cap C_j$.} \label{eq:fsum}
  \end{equation}
  Hence, $f(e) \neq 0$ and $f$ is a flow because it is a linear combination of flows.

  Now we recover the vertex variables $t_v$. Vertex $v$ is surrounded by three cycles; 
  for each $e \in E(v)$ we let $\ell(e)$ be the label of the cycle incident with~$v$ but not containing~$e$. 
  We put $t_v := \sum_{e \in E(v)} \ell(e)$. Because of \eqref{eq:fsum} we have 
  for every $e \in E(v)$ equality $t_v = \ell(e) + f(e)$, equivalently, $\ell(e) = t_v + f(e)$; 
  we will use it soon. 


  It remains to recover $\eps$. Its value is determined by $t$ but 
  it is not obvious that a choice satisfying system \eqref{eq:cond} exists.
  Let $uv \in E$. Fix $i$ so $uv \in C_i$ and $\rho_{u}(uv) \not\in C_i$, there
  exists exactly one such $i$. We define $\eps_{uv} = 0$ if $\rho_{v}(uv) \not\in C_i$
  and $\eps_{uv} = 1$ otherwise. The equation of the system~\eqref{eq:cond}
  containing $\eps_{uv}$ is:
  \begin{align*}
    (t_u + f(\rho_{u}(uv))) + (t_v + f(\rho_{v}(uv))) &= \eps_{uv}f(uv)
  \end{align*}
  The value $\ell(\rho_{u}(uv)) = t_u + f(\rho_{u}(uv))$ is the label of the cycle containing $uv$ and
  not containing $\rho_{u}(uv)$, and similarly for $v$. Hence they are the same
  if $C_i$ avoids both $\rho_{u}(uv)$ and $\rho_{v}(uv)$ and differ by $f(uv)$ if
  $C_i$ contains $\rho_{v}(uv)$.

  Injectivity: Consider two different
  $(f, (t, \eps)), (f', (t', \eps')) \in \calN_k$. If
  $f \neq f'$ then there is $e \in E$ such that $f(e) \neq f'(e)$ and the cycles on $e$ must
  differ because the flow is the sum of their labels. If $f = f'$ then
  there exists $v \in V$ such that $t(v) \neq t'(v)$
  ($\eps$ is uniquely determined by $t$
  so two solutions cannot differ only in $\eps$).
  But again the sum of the cycle labels
  around a vertex $v$ is $t_v$, so the cycles around $v$ are different.
\end{proof}

\begin{definition}\label{def:sigma}
  Let $G$ be a cubic graph.
  We denote $\CC_*$ the set of all CiDCs of $G$.
  We denote $\sigma: \cup_k \CC_{2^k} \to \CC_*$ a function that maps
  every $2^k$-CyDC of $G$ to the CiDC obtained by decomposing its cycles
  into circuits.\footnote{Because $G$ is cubic, this decomposition is unique.}
\end{definition}

\begin{observation}\label{obs:cidc}
  Let $G$ be a connected cubic graph, fix $\rho_v$ for every vertex~$v$ of~$G$. 
  Let $\calC \in \CC_{2^k}$ and $\calC' \in \CC_{2^{k'}}$.
  Then $\sigma(\calC) = \sigma(\calC')$ if and only
  $\eps = \eps'$ where
  $(f, (t, \eps)) = \mu_k^{-1}(\calC)$ and
  $(f', (t', \eps')) = \mu_{k'}^{-1}(\calC')$.
\end{observation}

\begin{proof}
  Let $uv = e \in E$. Let $\{a,b,e\} = E(u)$, $\{x,y,e\} = E(v)$,
  $a = \rho_{u}(e)$ and $x = \rho_{v}(e)$ as in Figure~\ref{fig:eps}.
  By Observation~\ref{obs:keyemb}, if $\eps_e = 0$ then the cycles containing $e$ contain consecutive 
  edges $a, e, x$ and $b, e, y$, respectively. 
  For $\eps_e = 1$ the transitions are $a, e, y$ and $b, e, x$. 
 
  Hence changing $\eps_e$ changes
  the circuit structure around $e$ ensuring that the corresponding CiDCs
  are different. On the other hand, if $\eps = \eps'$ then
  the circuit structure does not differ anywhere, so the CiDCs are the same.
\end{proof}

\begin{corollary}\label{col:cydc-to-cidc}
  Let $G$ be a connected cubic graph, fix $\rho_v$ for every vertex~$v$ of~$G$. 
  Let $(f, (t, \eps))$, $(f, (t', \eps')) \in \calN_k$.
  Then $\sigma(\mu_k(f, (t, \eps))) = \sigma(\mu_k(f, (t', \eps')))$
  if and only if there exists $\Delta \in \Z_2^k$ such that $t'_v = t_v + \Delta$
  for all $v \in V$.
\end{corollary}

\begin{proof}
  The CiDCs are the same if and only if $\eps = \eps'$
  by Observation~\ref{obs:cidc}, we show that this is equivalent to $t_v + t'_v$ being constant. 
  Forward implication first: adding~\eqref{eq:cond} for $(t,\eps)$ 
  and for $(t',\eps)$ gives us that $t_u + t'_u = t_v + t'_v$ for 
  any edge $uv$. As $G$ is connected, the conclusion follows. 
  The other way around: Replacing $t_v$ by $t'_v = t_v + \Delta$ for all $v \in V$
  in~\eqref{eq:cond} keeps the equations valid, which shows $\eps = \eps'$ is still 
  a solution, so the CiDCs are the same. Moreover, $\eps$ is determined by $t$, 
  so there cannot be another solution. 
\end{proof}

We finish this section with a positive result, that every CiDC can be obtained
from a flow for high enough $k$, and a related negative result that
the system (\ref{eq:cond}) is not always solvable for $k \geq 4$.

\begin{observation}
  Let $G$ be a bridgeless cubic graph on $n$ vertices.
  Let $k \geq \lceil \log_2 n \rceil$. Then $\sigma[\CC_{2^k}] = \CC_*$.
\end{observation}

\begin{proof}
  The inclusion $\sigma[\CC_{2^k}] \subseteq \CC_*$ is trivial.
  For the other one, consider any CiDC with $c$ circuits. As each of the circuits is of length at least~3, 
  we have $3c \le 2m$, so $c \le n \le 2^k$. Consequently, this CiDC comes from a $2^k$-CyDC: 
  we can keep each circuit as a separate cycle (and add empty cycles if needed). 
\end{proof}

This bound cannot be asymptotically improved because for infinitely many~$r$, there
exists an embedding of $K_r$ such that each face is a triangle (in fact, for 
every $r$ congruent to $0$ or $1$ modulo $3$, see \cite{Ringel}). 
So, the corresponding dual of $K_r$ is a cubic graph~$H$. The graph $H$ has $n = \Theta(r^2)$ vertices and this embedding
of it has $r = \Theta(\sqrt{n})$ faces such that each face shares an edge with
every other face. So we need at least $\Omega(\sqrt{n})$ colors and $k = \Omega(\log n)$.

We are going to discuss when a system of linear equations has a solution by checking, whether a linear combination
of them is of form $0=1$. To this end, we study required properties of the coefficients of such combination. 

\begin{definition}
  Let $f: E \to \Z_2^k$ be a flow on $G$. We say a flow $h: E \to \Z_2^k$ is its
  {\em annihilator flow}
  if $\ip{f(e)}{h(e)} = 0$ for all $e \in E$ where $\ip{x}{y}$ is the
  usual inner product defined as $\sum_{i=1}^k x_iy_i$.
  We denote $T_h = \{ v \in V : \forall e \in E(v), h(e) \neq 0 \}$.
\end{definition}

\begin{observation}\label{obs:annih}
  Let $f: E \to \Z_2^k$ be a NZ flow and $h: E \to \Z_2^k$ 
  the coefficients of a linear combination
  of equations of system (\ref{eq:cond}) for $f$. Then the left-hand side of this linear combination
  is zero if and only if $h$ is an annihilator flow of $f$.
\end{observation}

\begin{proof}
  The linear combination of the left-hand side is
  $$ \sum_{uv \in E} (\ip{t_u}{h(uv)} + \ip{t_v}{h(uv)} +
        \eps_{uv}\ip{f(uv)}{h(uv)}).
  $$
  It is zero if and only if the coefficients of all variables are zero:
  $$\sum_{e \in E(v)} h(e) = 0 \quad \forall v \in V
    \quad \textrm{and} \quad
    \ip{f(e)}{h(e)} = 0 \quad \forall e \in E.$$
  That is precisely definition of an annihilator flow of $f$.
\end{proof}

\begin{corollary}\label{cor:num-sols}
  Let $f: E \to \Z_2^k$ be a NZ flow such that the system~\eqref{eq:cond}
  is solvable and let $d$ be the dimension of the space of all annihilator flows of $f$.
  Then the solution space of system~\eqref{eq:cond} has dimension
  $d + n(3-k)/2$ (and hence $2^{d + n(3-k)/2}$ solutions).
\end{corollary}

\begin{proof}
  We interpret $t$ as $nk$ binary variables and $\eps$ as $\tfrac32n$ binary variables. 
  System~\eqref{eq:cond} has $(\tfrac32n)k$~binary equations. 
  The left kernel of the underlying binary matrix has dimension~$d$, so basic linear algebra gives us 
  the solution space of dimension $nk + \tfrac32n - (\tfrac32n)k + d$.
\end{proof}

\begin{definition}
  Let $f: E \to \Z_2^k$ be a NZ flow, $h: E \to \Z_2^k$ its annihilator flow,
  and $v \in V$.
  We define {\em vertex parity} $\lambda_v = \ip{f(\rho_{v}(e))}{h(e)}$ for any $e \in E(v)$.
\end{definition}

\begin{observation}\label{obs:annih-vert}
  Vertex parity is well-defined and independent of the choice of $\rho_v$. 
\end{observation}

\begin{proof}
  Let $\{e_1,e_2,e_3\}=E(v)$, indexed so that
  $\rho_{v}(e_i)=e_{i+1}$, with indices taken modulo~$3$. Observe that
  \begin{align*}
    \ip{f(e_2)}{h(e_1)}
      &= \ip{f(e_2)}{h(e_2)+h(e_3)}
       = \ip{f(e_2)}{h(e_3)} \\
      &= \ip{f(e_1)+f(e_3)}{h(e_3)}
       = \ip{f(e_1)}{h(e_3)}.
  \end{align*}
  Applying the same argument cyclically shows that the three values
  $\ip{f(\rho_{v}(e_i))}{h(e_i)}$ are equal.
\end{proof}

\begin{observation}\label{obs:annih-3}
  Let $f: E \to \Z_2^k$ be a NZ flow and $h: E \to \Z_2^k$ its annihilator flow.
  Then the right-hand side of the linear combination of system \eqref{eq:cond}
  with coefficients $h$ is zero if and only if
  $\sum_{v \in T_h} \lambda_v = 0$.
\end{observation}

\begin{proof}
  We reorder the sum and use vertex parity
  $$ 
  \sum_{uv \in E} \bigl(\ip{f(\rho_{u}(uv))}{h(uv)} + \ip{f(\rho_{v}(uv))}{h(uv)}\bigr)
    = \sum_{v \in V} \sum_{e \in E(v)} \ip{f(\rho_{v}(e))}{h(e)}
    = \sum_{v \in V \setminus T_h} \lambda_v + \sum_{v \in T_h} \lambda_v
  $$
  If $v \not\in T_h$, then there is $e \in E(v)$ such that $h(e) = 0$, so $\lambda_v = 0$.
  What remains is the sum over $v \in T_h$.
\end{proof}

\begin{corollary}
The system~\eqref{eq:cond} is solvable if and only if
$
  \sum_{v \in T_h} \lambda_v = 0
$
for every annihilator flow $h$.
\end{corollary}

\section{Applications}
\label{sec:appl}

In this section we show several applications of the general technique developed so far. 

\subsection{3-Edge-Colorable Cubic Graphs and $k = 2$}\label{sec:kfour}

For this section we fix $G$ to be a 2-connected cubic 3-edge-colorable graph.
By Theorem~\ref{thm:4NZF}, it has a nowhere-zero flow $f: E \to \Z_2^2$.
(It can be obtained as interpreting the three edge colors as non-zero
elements of $\Z_2^2$.)
The fact that 3-edge-colorable cubic graphs have a CyDC goes back to
Szekeres~\cite{szekeres}.

\begin{theorem}[Tutte~\cite{tutte49}]\label{thm:4NZF}
  Every 3-edge-colorable cubic graph has NZ $\Z_2^2$-flow.
\end{theorem}

In this section, we make the choice of $\rho_v$ dependent on the flow, or 
rather, on the edge-coloring $c : E \to \Z_3$; 
we choose each $\rho_v$ so that $c(\rho_v(e)) = c(e)+1 \pmod 3$. 
Then we have $f(\rho_{u}(uv)) = f(\rho_{v}(uv))$ (as $f$ is a lift of~$c$). 
For this choice, the right-hand side of system (\ref{eq:cond}) is zero,
so it is trivially solvable.
Define linear function $\alpha: \Z_2^2 \to \Z_2^2$ via formula
$\alpha(x_1, x_2) = (x_2, x_1)$.
Observe that $\ip{x}{\alpha(x)} = 0$ and let $\alpha(f)(e) = \alpha(f(e))$.

\begin{observation}\label{obs:annih-k2}
  Let $f$ be a NZ $\Z_2^2$-flow on a connected cubic graph $G$.
  Then there exist only two annihilator flows: zero and $\alpha(f)$.
\end{observation}

\begin{proof}
  Zero is obviously an annihilator flow. Consider $h$ an annihilator
  flow that is non-zero on some edge $e$. The value $h(e)$ must be $\alpha(f(e))$
  because it is the only non-zero value such that $\ip{h(e)}{f(e)} = 0$.
  Look at vertex $v$ such that $e \in E(v)$. At least one other edge $e_2$ of $E(v)$
  must have $h(e_2)$ non-zero, so $h(e_2) = \alpha(f(e_2))$. Consequently, even 
  the third edge must have non-zero $h$. This propagates through the whole graph
  due to connectivity.
\end{proof}

\begin{observation}\label{obs:num-4cydc}
  Let $f$ be a NZ $\Z_2^2$-flow on a connected cubic graph $G$. Then the construction
  above gives exactly $2^{n/2 + 1}$ different 4-CyDCs of $G$.
\end{observation}

\begin{proof}
  By Corollary~\ref{cor:num-sols} and Observation~\ref{obs:annih-k2},
  the dimension of the space of solutions is $1 + n(3 - 2)/2 = 
  n/2 + 1$.
\end{proof}

\begin{proof}[Proof of Theorem~\ref{thm:3-col}]
Fix a proper 3-edge-coloring of $G$ and the corresponding NZ $\Z_2^2$-flow $f$.
By Observation~\ref{obs:num-4cydc}, the solutions of system (\ref{eq:cond}) give
$2^{n/2+1}$ pairwise different 4-CyDCs. And only four CyDCs are
mapped to the same CiDC due to Corollary~\ref{col:cydc-to-cidc}.
Hence the number of CiDCs of $G$ is at least $2^{n/2+1} / 4 = 2^{n/2-1}$.
\end{proof}

\subsection{General cubic graphs and $k = 3$}\label{sec:existence}

We fix $k = 3$ and remind the reader of existence Theorem~\ref{thm:8NZF}.

\begin{theorem}[Jaeger~\cite{jaeger}, Kilpatrick~\cite{kilpatrick}]\label{thm:8NZF}
  Every cubic bridgeless graph has nowhere-zero $\Z_2^3$-flow.
\end{theorem}

The following lemma is a reformulation of part of OpenAI proof~\cite{openai}.

\begin{observation}\label{obs:align}
  For every $f \in \mathcal F_3$,
  the system of equations (\ref{eq:cond}) has a solution.
\end{observation}

\begin{proof}
  By Observation~\ref{obs:annih},
  it is enough to show that for every annihilator flow $h$, the right-hand side is
  zero. By Observation~\ref{obs:annih-3}, it is enough to check that
  $\sum_{v \in T_h} \lambda_v = 0.$

  Let $v \in T_h$. Let $\{ e_1, e_2, e_3 \} = E(v)$.
  If $\lambda_v = 0$ then $\ip{h(e_i)}{f(e_j)} = 0$ for all $i, j$. But
  there is only one non-zero value $x$ for which $\ip{x}{f(e_j)} = 0$ holds for all $j$
  (we work in $\Z_2^3$ and $\ip{x}{f(e_j)}$ for all $j$ are three linear equations
  with one linear dependency, so the solution space has dimension one).
  Hence all $h(e_i) = x$ and the sum around vertex $v$ is $3x = x \neq 0$
  which is a contradiction with definition of a flow. Hence $\lambda_v = 1$.

  Notice that $T_h$ are exactly vertices of degree three after deleting edges with
  $h(e) = 0$ and there are no vertices of degree one because $h$ is a flow.
  By hand-shaking lemma, the size of $T_h$ is even and the final sum is zero.
\end{proof}

It seems unlikely that there will be many solutions for a single flow. On the
other hand, it is known that every bridgeless graph has many NZ $\Z_2^3$-flows.
Given the existence of a single NZ $\Z_2^3$-flow (guaranteed by Theorem~\ref{thm:8NZF}) 
it is easy to show that there are many of them:

\begin{observation}\label{obs:many-8NZF-triv}
  If a cubic graph $G$ has a NZ $\Z_2^3$-flow, it has $\Omega(2^{2n/7})$
  NZ $\Z_2^3$-flows.
\end{observation}

\begin{proof}
  Let $f$ be the NZ $\Z_2^3$-flow. Let $z$ be its least used non-zero value.
  The number of edges with $f(e) = z$ is at most $m / 7$. So after removing
  all such edges, at least $6m/7$ edges remain. This subgraph has cycle space
  $\calC$ of dimension at least $9n/7 - n + 1 = 2n/7 + 1$. For any cycle
  $C \in \calC$ define a $\Z_2^3$-flow~$g$ with value~$z$ on~$C$ (and 0 elsewhere).
  Then $f+g$ is a NZ flow and different choices of~$C$ give different flows. 
\end{proof}

A better bound is known due to DeVos, Langhede, Mohar and Šámal
(unpublished as of July 2026). 

\begin{theorem}[DeVos, Langhede, Mohar and Šámal~\cite{dlms}, Theorem~1.9]
  \label{thm:many-8NZF}
  Every 3-connected cubic graph has at least $2^{n/2 - 1}$ NZ $\Z_2^3$-flows.
\end{theorem}

\begin{corollary}\label{cor:8CyDC}
  Every 3-connected cubic graph has at least $2^{n/2 + 2}$ 8-CyDCs.
\end{corollary}

\begin{proof}
  Every flow gives at least one CyDC. Moreover for $(f, (t, \eps)) \in \calN_3$
  and every $x \in \Z_2^3$ the tuple $(f, (t + x, \eps))$ (i.e., shifting
  all $t_v$ by $x$) is also in $\calN_3$. (Corresponding CyDCs differ only by renaming
  cycles by $x$.) So every flow gives at least 8 CyDCs.
\end{proof}

The question remains how many 8-CyDCs can correspond to the same CiDC. This can
be seen as a bipartite graph with partition $A = \CC_8$,
partition $B = \CC_*$
and $$E = \{ (\calC, \calC') \in \CC_8 \times \CC_* :
  \sigma(\calC) = \calC' \}.$$
The degrees of vertices in $A$ are always one.
We also have a good lower bound on $|A|$. So it remains to upper-bound
degrees in $B$. Note that some vertices in $B$ might have degree zero as they
may not be colorable by 8 colors so they are not counted by the following bounds.

\begin{observation}\label{obs:cidc-color}
  A CiDC $\calC'$ consisting of $c$ circuits can be obtained from $\mathcal O(6^c)$
  8-CyDCs.
\end{observation}

\begin{proof}
We are interested in the number of proper 8-colorings of circuits of $\calC'$
where two circuits cannot have the same color if they share an edge.
The CiDC $\calC'$ gives a strong embedding of $G$ on some surface
such that the circuits are its faces (take $G$ and glue a disk around every circuit in $\calC'$).
Taking dual of this embedding we obtain a (multi)graph $D$ whose vertices correspond to the 
circuits of $\calC'$ and vertices are connected by an edge if and only if
the two circuits share an edge.

We want to count 8-colorings of $D$. Because $G$ is cubic, $D$ is a triangulation.
Fix a spanning tree $T$ of $G$ and root it in a vertex $v_1$. Repeatedly remove
a leaf of $T$ until only $v_1$ remains. Let $v_n, \dots, v_2$ be the order in which
the leaves were removed.

Now we color vertices of $D$ -- faces of $G$.
Start with the three faces around $v_1$. This gives $8\cdot7\cdot6$ options.
Then for $i = 2,\dots,n$ we look at the three (pairwise adjacent) 
faces of $G$ around $v_i$.
Note that $v_i$ is connected by an edge of~$T$ to some $v_j$ for $j < i$ and hence it
shares at least two faces with it and these faces share an edge $v_iv_j$ so
they have different colors. So there is at most one face to color and there
are at most 6 colors available. So the total bound is $8\cdot7\cdot6 \cdot 6^{c-3}
= \mathcal O(6^c)$ because each face is colored at most once.
\end{proof}

\begin{proof}[Proof of Theorem~\ref{thm:snarks}]
  To use Observation~\ref{obs:cidc-color}, we need to upper-bound the number
of circuits in a CiDC. Using the girth $g$ of $G$ we obtain 
an easy bound $c \leq 3n / g$. Then, using Corollary~\ref{cor:8CyDC} we get
\begin{align*}
  |\CC_*| \geq |\CC_8| / \mathcal O(6^{3n/g})
    &= \Omega(2^{n/2 - (3n/g) \cdot \log_2 6}). 
\end{align*}
As $g > 6 \log_2 6 \approx 15.51$, the exponent $n/2 - (3n/g) \cdot \log_2 6$ is positive. 
So, for graphs with girth at least 16, the exponent is $cn$ for some $c > 0$.
\end{proof}

Replacing Theorem~\ref{thm:many-8NZF} (which is not peer reviewed yet) with
Observation~\ref{obs:many-8NZF-triv} would weaken Theorem~\ref{thm:snarks}
to require girth $g > 10.5 \log_2 6 \approx 27.14$.

\subsection{Planarity characterization}
\label{sec:planchar} 

\begin{theorem}[Mac Lane~\cite{maclane}]\label{thm:maclane}
  A graph $G$ is planar if and only if it has basis of cycle space such that
  each edge is in at most two of its elements.
\end{theorem}

For a planar 2-connected graph, one CyDC is given by face boundaries, in fact this 
was an inspiration for Conjecture~\ref{conj:CDC} in the first place. Next, we show that we 
can actually get a CyDC by our construction for any NZ $\Z_2^k$-flow. 

\begin{observation}\label{obs:plan}
  Let $G$ be a planar 2-connected cubic graph. Then the system (\ref{eq:cond})
  is solvable for every NZ $\Z_2^k$-flow.
\end{observation}

\begin{proof}
  Let $(C_i)_{i=1}^l \subseteq \Z_2^E$ be the basis of cycle space of $G$ given by
  Theorem~\ref{thm:maclane}.
  As $G$ is bridgeless, every edge is covered by some cycle, thus it is covered once or twice by the basis. Note that the edges covered once
  are a cycle (as it is the symmetric difference $C_1 \oplus \cdots \oplus C_l$) and denote it by~$C_0$.
  Let $f: E \to \Z_2^k$ be a NZ flow. 
  Using $k$ times the fact that $(C_i)_{i=1}^l$ is a basis, we get 
  $\alpha_i \in \Z_2^k$ such that $f = \sum_{i=1}^l \alpha_i \gamma(C_i)$.
  Letting $\alpha_0 = 0$, we can extend the summation from $i=0$ to $l$.
  We claim that for every $i \neq j$ if $C_i \cap C_j \neq \emptyset$ then $\alpha_i \neq \alpha_j$.  
  (If there is $e \in C_i \cap C_j$ then $\alpha_i + \alpha_j = f(e) \neq 0$.)
  So $\alpha$ is a proper coloring of cycles with $\Z_2^k$ colors
  and gives a $2^k$-CyDC $\calC$ with $a$-th cycle being the union of
  all $C_i$ with $\alpha_i = a$. This CyDC is generated by flow $f$.
  Bijectivity of $\mu_k$ implies $\calC = \mu_k(f,s)$ for some $s$, the claimed solution of \eqref{eq:cond}.
\end{proof}

Experimental data suggest that for non-planar graphs and $k \geq 4$ there
always exists a flow for which the system is not solvable.
We are not able to prove that, so we prove a weaker version -- there exists
a non-solvable flow for $k = 2 + z$ where $z$ is the number of dimensions
needed for NZ flow (so 2 for 3-edge-colorable graphs and 3 otherwise).

\begin{figure}[h]
  \centering
  \begin{tikzpicture}[scale=1.15,
    vtx/.style ={circle,fill=black,inner sep=0pt,minimum size=5pt},
    edg/.style ={black!70,line width=.9pt},
    every node/.style={font=\small}]

    \definecolor{colA} {RGB}{215, 45, 35}   
    \definecolor{colB} {RGB}{150, 20, 25}   
    \definecolor{colAp}{RGB}{240,150, 20}   
    \definecolor{colBp}{RGB}{ 25,140, 60}   

    \newcommand{\fourcycle}[7]{%
      \draw[#1,line width=1.7pt,opacity=.9,rounded corners=7pt]
        ($(#3)!#2!(#7)$) -- ($(#4)!#2!(#7)$) --
        ($(#5)!#2!(#7)$) -- ($(#6)!#2!(#7)$) -- cycle;}

    \coordinate (a) at (0   , 0   );   
    \coordinate (c) at (6.9 , 0.1 );   
    \coordinate (t) at (6.3 ,-3.2 );   
    \coordinate (x) at (3   , 2.1 );   
    \coordinate (y) at (3.1 , 0.1 );   
    \coordinate (z) at (2.7 ,-2.3 );   

    \foreach \u in {a,c,t}
      \foreach \v in {x,y,z}
        \draw[edg] (\u) -- (\v);

    \coordinate (mA)  at (barycentric cs:a=1,x=1,c=1,y=1);
    \coordinate (mB)  at (barycentric cs:a=1,y=1,c=1,z=1);
    \coordinate (mAp) at (barycentric cs:a=1,y=1,z=1);
    \coordinate (mBp) at (barycentric cs:a=1,x=1,c=.8,y=1);

    \fourcycle{colA} {3mm}{a}{x}{c}{y}{mA}    
    \fourcycle{colB} {5mm}{a}{y}{c}{z}{mB}    
    \fourcycle{colBp}{5mm}{a}{x}{t}{y}{mBp}   
    \fourcycle{colAp}{3mm}{a}{y}{t}{z}{mAp}   

    \node[vtx,label={[label distance=1pt]left :$a$}] at (a) {};
    \node[vtx] at (c) {};
    \node[vtx,label={[label distance=1pt]right:$t$}] at (t) {};
    \node[vtx] at (x) {};
    \node[vtx] at (y) {};
    \node[vtx] at (z) {};

    \node[colA ] at (5.30, 0.45) {$A$};
    \node[colB ] at (5.35,-0.30) {$B$};
    \node[colBp] at (6.15,-2.45) {$B'$};
    \node[colAp] at (4.30,-2.30) {$A'$};
  \end{tikzpicture}
  \caption{Subdivision of $K_{3,3}$ with circuits $A$, $B$, $A'$ and $B'$}
  \label{fig:k33}
\end{figure}
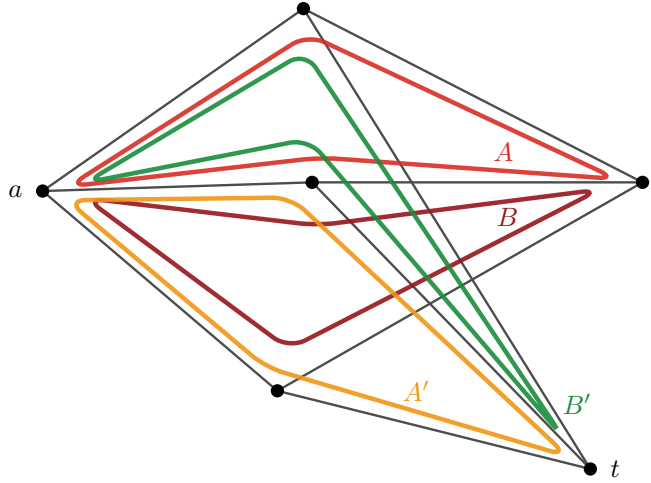

The idea behind the following construction is that the subspace of
the cycle space generated by circuits of a CiDC cannot contain both
$A$ and $B$ from Figure~\ref{fig:k33}.

\begin{observation}\label{obs:plus2nonsolv}
  Let $G$ be a non-planar cubic graph, and let $f': E \to \Z_2^k$ be a NZ flow.
  Then there exists a NZ flow $f: E \to \Z_2^{k+2}$ such that the system~\eqref{eq:cond}
  for $f$ is not solvable.
\end{observation}

\begin{proof}
  By Kuratowski theorem~\cite{kuratowski}, $G$ contains subdivision of either
  $K_{3,3}$ or $K_5$ and it cannot contain subdivision of $K_5$ because
  $K_5$ has vertices of degree 4 and $G$ is cubic. Take the subdivision of $K_{3,3}$ together with
  circuits denoted in Figure~\ref{fig:k33}.

  Define $f: E \to \Z_2^{k+2}$ to be $f'$ on the first $k$ coordinates, $\gamma(A)$ on coordinate
  $k+1$ and $\gamma(B)$ on coordinate $k+2$. Define $h: E \to \Z_2^{k+2}$ to be zero
  on the first $k$ coordinates, $\gamma(A')$ on coordinate $k+1$, and $\gamma(B')$
  on coordinate $k+2$. Observe that $h$ is annihilator flow for $f$.

  Now let us sum the right-hand side of system \eqref{eq:cond}. By
  Observation~\ref{obs:annih-3}, it is enough to look at vertices in $T_h$,
  and there are only two such vertices -- $a$ and $t$. Obviously $\lambda_t = 0$.
  Denote by $e_A$, $e_{AB}$ and $e_B$ edges adjacent to $a$ and lying in only $A$, both
  $A$ and $B$ and only $B$. Then $\ip{f(\rho_{a}(e_{AB}))}{h(e_{AB})} = \lambda_a = 1$
  because $h(e_{AB})$ has non-zero exactly the last two coordinates and both $e_A$ and $e_B$
  have exactly one of the last two coordinates non-zero.
  So $\sum_{v \in T_h} \lambda_v = 1$ and the system is not solvable.
\end{proof}

Using well-known results about existence of nowhere-zero flow --
Theorems~\ref{thm:4NZF} and~\ref{thm:8NZF} -- we obtain
the desired results (note that 3-edge-colorable graph must be bridgeless).

\begin{corollary}
  Let $G$ be 3-edge-colorable non-planar cubic graph.
  Then there exists NZ $\Z_2^4$-flow for which the system \eqref{eq:cond} is not
  solvable.
\end{corollary}

\begin{corollary}\label{cor:nonplan5}
  Let $G$ be bridgeless non-planar cubic graph.
  Then there exists NZ $\Z_2^5$-flow for which the system \eqref{eq:cond} is not
  solvable.
\end{corollary}

\begin{proof}[Proof of Theorem~\ref{thm:planchar}]
  If $G$ is planar, Observation~\ref{obs:plan} shows that system~\eqref{eq:cond} 
  is solvable for every nowhere-zero $\Z_2^k$-flow and every $k$. Conversely, 
  if $G$ is nonplanar, Corollary~\ref{cor:nonplan5} gives a nowhere-zero $\Z_2^5$-flow for which the system is not solvable.
  If $G$ is 3-edge-colorable, Observation~\ref{obs:plus2nonsolv} gives such a flow already in $\Z_2^4$.
\end{proof}

\subsection{The 5-CyDC Conjecture}
\label{sec:5CDC} 

The purpose of this section is to express Conjecture~\ref{conj:5CDC}
within our flow framework. We first collect a system of equations whose
solvability is equivalent to the existence of a 5-CyDC. We then eliminate
the auxiliary variables from this system and obtain an equivalent condition
on the NZ $\Z_2^3$-flow alone.

Our construction has $2^k$ labeled cycle slots. We therefore view a
5-CyDC as an 8-CyDC with three empty cycles. Since the cycles are labeled,
we may relabel any such cover so that its nonempty cycles have labels in
$\calI=\{000,001,010,011,100\}$. 
Conversely, discarding the three empty cycles of an 8-CyDC supported
on~$\calI$ gives a 5-CyDC.

Fix a NZ $\Z_2^3$-flow $f$ and put
\[
  F=\{e\in E:\ip{f(e)}{100}=1\},\qquad
  M=f^{-1}(100),\qquad
  Z=\{v\in V:E(v)\cap M\ne\emptyset\}.
\]
The set $F$ is a cycle, since its characteristic vector is the first
coordinate of~$f$. Thus, for every $v\in V(F)$, there is a unique edge
$e_v^*\in E(v)\setminus F$. Moreover, $M$ is a matching: if two edges
at a vertex had value $100$, the flow equation would force the third
edge to have value~$000$.

For a flow $f$, let $\mathsf{Five}(f)$ denote the following system
in variables $t_v\in\Z_2^3$ for $v\in V$ and
$\eps_e\in\Z_2$ for $e\in E$:
\begin{align}
  \left.
  \begin{aligned}
    t_u+t_v+\eps_{uv}f(uv)
      &=f(\rho_{u}(uv))+f(\rho_{v}(uv))
        &&\quad \forall uv\in E,\\
    t_v&=100+f(e_v^*)
        &&\quad \forall v\in V(F),\\
    \ip{t_v}{100}&=0
        &&\quad \forall v\in V\setminus V(F).
  \end{aligned}
  \qquad\right\} \label{eq:five-system}
\end{align}
The first line is precisely the alignment system~\eqref{eq:cond};
the other two lines exclude the labels $101$, $110$, and $111$ as we show now. 

\begin{observation}[Equational formulation]\label{obs:5cydc-equations}
  For a NZ $\Z_2^3$-flow $f$, the solutions of
  $\mathsf{Five}(f)$ correspond bijectively under~$\mu_3$ to the
  8-CyDCs that induce~$f$ and are supported on~$\calI$.
  Consequently, $G$ has a 5-CyDC if and only if
  $\mathsf{Five}(f)$ is solvable for some NZ $\Z_2^3$-flow~$f$.
\end{observation}

\begin{proof}
Suppose first that $(t,\eps)$ satisfies
$\mathsf{Five}(f)$. The three labels of cycles incident with a vertex
$v$ are
$
  \{t_v+f(e):e\in E(v)\}.
$
If $v\in V(F)$, then the second line of~\eqref{eq:five-system} gives
$t_v+f(e_v^*)=100$. The other two labels have first coordinate zero,
because they correspond to the two edges of~$F$ incident with~$v$.
If $v\notin V(F)$, then both $t_v$ and all three values $f(e)$,
$e\in E(v)$, have first coordinate zero. Hence every cycle label
belongs to~$\calI$.

Conversely, suppose that $(f,(t,\eps))\in\calN_3$ and that
$\mu_3(f,(t,\eps))$ is supported on~$\calI$. The three labels
at every vertex are distinct, since the three incident values of a
NZ flow on a cubic graph are distinct. Let $v\in V(F)$. If $t_v$ had
first coordinate zero, then the two labels $t_v+f(e)$ for
$e\in E(v)\cap F$ would be distinct and would both have first
coordinate one. This is impossible because $100$ is the only such
label in~$\calI$. Thus $t_v$ has first coordinate one, and so does
$t_v+f(e_v^*)$; it must therefore equal~$100$. If
$v\notin V(F)$ and $t_v$ had first coordinate one, all three distinct
labels at~$v$ would have first coordinate one, which is again
impossible. This proves the last two lines of~\eqref{eq:five-system}.

It remains only to pass between five and eight labels. Given a
5-CyDC, append three empty cycles and label the eight slots so that
the five original cycles have labels in~$\calI$. By
Observation~\ref{obs:cydc-to-sol}, this labeled 8-CyDC has a unique
preimage $(f,(t,\eps))$, and the preceding paragraph shows
that it solves $\mathsf{Five}(f)$. The converse follows by discarding
the three cycles with labels outside~$\calI$.
\end{proof}

For a fixed flow, system~\eqref{eq:five-system} has an especially
simple solvability criterion.

\begin{theorem}\label{thm:5cydc}
  Let $f$ be a NZ $\Z_2^3$-flow, and define $F$, $M$, and $Z$ as above.
  The following are equivalent:
  \begin{enumerate}
    \item $\mathsf{Five}(f)$ has a solution;
    \item every component of $(V,E\setminus F)$ contains an even
      number of vertices from~$Z$;
    \item there is an edge set $P\subseteq E\setminus F$ whose
      odd-degree vertices are precisely the vertices in~$Z$.
  \end{enumerate}
  If these conditions hold and $(V,E\setminus F)$ has $c$ components,
  then the number of solutions of $\mathsf{Five}(f)$ is
  $2^{\,|E\setminus F|-|V|+c}$.
\end{theorem}

We prove the theorem by reducing system~\eqref{eq:five-system} and
then describing all linear dependencies among the reduced equations.
Put
\[
  U=V\setminus V(F),\qquad W=V(F),\qquad E^*=E\setminus F,
\]
and let $q:E\to\Z_2^2$ be the projection of~$f$ onto its last two
coordinates. For $v\in U$, let $t_v^*$ denote the last two
coordinates of~$t_v$. If $v\in W$, the second line
of~\eqref{eq:five-system} already fixes these coordinates as
$q(e_v^*)$.

First consider an edge $e=uv\in F$. Substituting the forced values
of $t_u$ and $t_v$ into its alignment equation gives
\[
  \eps_e f(e)
  =f(\rho_{u}(e))+f(e_u^*)+f(\rho_{v}(e))+f(e_v^*).
\]
For $w\in\{u,v\}$, the value
$f(\rho_{w}(e))+f(e_w^*)$ is either $0$ or $f(e)$. Hence the right-hand
side is either $0$ or $f(e)$, and the equation determines
$\eps_e$ uniquely.

It remains to consider $e=uv\in E^*$. The first coordinate of its
alignment equation is automatically satisfied. After interchanging
$u$ and $v$ when needed, the last two coordinates give the following
\emph{reduced system}:
\begin{align}
  \left.
  \begin{aligned}
    t_u^*+t_v^*+\eps_e q(e)
      &=q(\rho_{u}(e))+q(\rho_{v}(e))
        && e = uv \in E^*, u,v\in U,\\
    t_u^*+\eps_e q(e)
      &=q(\rho_{u}(e))+q(\rho_{v}(e))+q(e)
        && e = uv \in E^*, u\in U,\ v\in W,\\
    \eps_e q(e)
      &=q(\rho_{u}(e))+q(\rho_{v}(e))
        && e = uv \in E^*, u,v\in W.
  \end{aligned}
  \qquad\right\} \label{eq:reduced-five}
\end{align}
Solutions of~\eqref{eq:reduced-five} correspond bijectively to solutions of $\mathsf{Five}(f)$.

A linear combination of the reduced equations is specified by coefficients
$h:E^*\to\Z_2^2$, where the equation for~$e$ is multiplied by~$h(e)$.
Its left-hand side vanishes identically if and only if
\begin{align}
  \left.
  \begin{aligned}
    \sum_{e\in E(v)}h(e)&=0
      &&\quad \forall v\in U,\\
    \ip{h(e)}{q(e)}&=0
      &&\quad \forall e\in E^*.
  \end{aligned}
  \qquad\right\} \label{eq:reduced-annihilator}
\end{align}
We call such an~$h$ a \emph{reduced annihilator} (flow condition is not imposed).  

\begin{observation}\label{obs:reduced-annihilators}
  Let $K$ be a component of $(V,E^*)$, and recall that
  $\alpha(x_1,x_2)=(x_2,x_1)$. On $E(K)$, every reduced annihilator
  is either identically zero or satisfies
  \[
    h(e)=\alpha(q(e))\qquad\forall e\in E(K).
  \]
  Consequently, the space of reduced annihilators has dimension
  $c$, where $c$ is the number of components of $(V,E^*)$.
\end{observation}

\begin{proof}
For every $e\in E^*$, the value $q(e)$ is nonzero. The second
condition in~\eqref{eq:reduced-annihilator} therefore implies
\begin{equation}\label{eq:he}
  h(e)\in\{0,\alpha(q(e))\}.
\end{equation}
At a vertex $v\in U$, the three incident values of~$q$ are the three
nonzero elements of~$\Z_2^2$. Hence the first condition
in~\eqref{eq:reduced-annihilator} holds precisely when either all
three incident values of~$h$ are zero or all three equal the
corresponding values of~$\alpha\circ q$. This choice propagates
through each component. If a component contains no vertex of~$U$,
then it consists of a single edge with both endpoints in~$W$, and
the conclusion follows from~\eqref{eq:he}. The choices on distinct
components are independent.
\end{proof}

\begin{proof}[Proof of Theorem~\ref{thm:5cydc}]
The proof is a variant of the proof of~Observation~\ref{obs:annih-3}. 
By Observation~\ref{obs:reduced-annihilators}, the reduced system is
solvable if and only if its right-hand side sums to zero against the
reduced annihilator supported on each component~$K$. For this
annihilator, the sum is
\begin{align*}
  \sum_{v\in U\cap V(K)}
    \sum_{e\in E(v)}
      \ip{q(\rho_{v}(e))}{\alpha(q(e))}
  +
    \sum_{v\in W\cap V(K)}
      \ip{q(\rho_{v}(e_v^*))+q(e_v^*)}
         {\alpha(q(e_v^*))}.
\end{align*}
At a vertex $v\in U$, the values $q(\rho_{v}(e))$ and $q(e)$ are
distinct nonzero elements of~$\Z_2^2$, so every inner product in the
first sum equals~$1$. There are three of them, and hence $v$
contributes~$1$.

Now let $v\in W$. The two edges of~$F$ incident with~$v$ have last
two coordinates $q(\rho_{v}(e_v^*))$ and
$q(\rho_{v}(e_v^*))+q(e_v^*)$. One of these edges has value $100$ if and
only if
$
  q(\rho_{v}(e_v^*))\in\{00,q(e_v^*)\}
$.
Equivalently, the contribution of~$v$ to the second sum is zero if
$v\in Z$ and one otherwise. The entire sum for~$K$ is therefore,
modulo two,
\[
  |U\cap V(K)|+|(W\setminus Z)\cap V(K)|
  =|V(K)|-|Z\cap V(K)|.
\]
Every vertex has odd degree in~$K$: vertices in~$U$ have degree
three and vertices in~$W$ have degree one. The handshaking lemma
thus shows that $|V(K)|$ is even. Consequently, the right-hand side
vanishes for~$K$ if and only if $|Z\cap V(K)|$ is even. This proves
the equivalence of (1) and~(2).

The equivalence of (2) and (3) is the standard criterion for a
$Z$-join. Necessity follows from the handshaking lemma in each
component of $(V,E^*)$. For sufficiency, pair the vertices of~$Z$
within each component and take the symmetric difference of paths
joining the paired vertices.

Finally, the reduced system has $2|U|+|E^*|$ variables in~$\Z_2$ and
$2|E^*|$ equations over~$\Z_2$. Its space of dependencies has dimension
$c$, so, whenever the system is solvable, its solution space has
dimension
\[
  2|U|+|E^*|-(2|E^*|-c)
  =2|U|-|E^*|+c
  =|E^*|-|V|+c.
\]
The last equality follows from
$2|E^*|=3|U|+|W|$. This proves the claimed number of solutions.
\end{proof}

Hoffmann-Ostenhof characterized when a prescribed cycle is contained
in a 5-CyDC. Here ``contains'' means that the prescribed cycle is a
subgraph of one member of the cover; it need not itself be a member.

\begin{theorem}[Hoffmann-Ostenhof~\cite{hoffmann}]\label{thm:hoffmann}
  Let $G$ be a cubic graph and let $C_0$ be a cycle of~$G$.
  Then $G$ has a 5-CyDC containing $C_0$ if and only if there are a
  matching $M$ and cycles $C_1,C_2$ such that
  \begin{enumerate}
    \item $G-M$ has a NZ $\Z_2^2$-flow, and
    \item $C_1\cap C_2=M$ and $C_0\subseteq C_1$.
  \end{enumerate}
\end{theorem}

Let us explain the relation of our result to Theorem~\ref{thm:hoffmann}. Given a
$Z$-join~$P$ in Theorem~\ref{thm:5cydc}, take $C_1=F$, $C_2=M\cup P$.
The odd-degree vertices of both $M$ and $P$ are exactly~$Z$, so
$C_2$ is a cycle; moreover, $C_1\cap C_2=M$. Since $f$ is
nowhere-zero, $q^{-1}(00)=M$, and hence $q$ restricts to a NZ
$\Z_2^2$-flow on~$G-M$. This fits Theorem~\ref{thm:hoffmann} with $C_0=C_1=F$.
His construction consequently keeps $F$ as a member of the resulting
5-CyDC. What it does not specify is whether the resulting 
CyDC is of form $\mu_3(f,s)$ for some $s$, which is what we prove in Theorem~\ref{thm:5cydc}. 

Combining Observation~\ref{obs:5cydc-equations} and
Theorem~\ref{thm:5cydc} gives the promised flow-only version of the
5-CyDC conjecture.

\begin{conjecture}[Equivalent flow formulation]\label{conj:5cydc-flow}
  Every bridgeless cubic graph has a NZ $\Z_2^3$-flow~$f$ such that,
  with
  \[
    F=\{e\in E:\ip{f(e)}{100}=1\}
    \quad\text{and}\quad
    Z=\{v\in V:\exists e\in E(v),\ f(e)=100\},
  \]
  every component of $(V,E\setminus F)$ contains an even number of
  vertices from~$Z$.
\end{conjecture}

Conjecture~\ref{conj:5cydc-flow} is equivalent to
Conjecture~\ref{conj:5CDC}, not merely sufficient for it. In one
direction, the parity condition and Theorem~\ref{thm:5cydc} produce a
solution of~\eqref{eq:five-system}, hence a 5-CyDC. In the other
direction, relabel any 5-CyDC by~$\calI$ and use the flow induced by
its labels; Theorem~\ref{thm:5cydc} then gives the parity condition.
Theorem~\ref{thm:8NZF} guarantees the existence of NZ
$\Z_2^3$-flows, but the remaining problem is to choose one with this
additional parity property.

\section*{Acknowledgments} 
We have used GPT 5.6 and Claude Fable for draft versions of the paper.

\end{document}